\documentclass[12pt,oneside,reqno]{amsart}
\usepackage[numbers]{natbib}
 
\usepackage{color}
\usepackage{amsmath, amssymb,geometry,mathtools,etoolbox}
\usepackage{amsfonts}
\usepackage{mathrsfs}
\usepackage[arrow,matrix,curve,cmtip,ps]{xy}
\usepackage{tikz-cd}
\usepackage{bm}
\usepackage[english]{babel}
\usepackage{natbib}
\usepackage[T1]{fontenc}
\usepackage{epsfig}

\usepackage{amsthm}

\renewcommand{\ker}{\operatorname{Ker}}

\newcommand{\ad}{\mathop{Ad}}

\allowdisplaybreaks

\usepackage{xfrac}

\usepackage{tikz, tikz-cd}
\usepackage{lipsum}

\usepackage{adjustbox}
\usepackage{verbatim}
\usepackage{graphicx}
\usepackage{enumerate}
\usepackage{dcolumn,array}
\usepackage{mathtools, caption} 

\newcommand{\righttext}[1]{\qquad\text{#1 }}
\newcommand{\midtext}[1]{\qquad\text{#1 }\qquad}

\usepackage{relsize}
\usetikzlibrary{arrows}

\DeclarePairedDelimiterX{\norm}[1]{\lVert}{\rVert}{#1}

\newtheorem{theorem}{Theorem}[section]
\newtheorem{lemma}[theorem]{Lemma}
\newtheorem{proposition}[theorem]{Proposition}
\newtheorem{corollary}[theorem]{Corollary}

\newtheorem*{theorem*}{Theorem}
\theoremstyle{remark}
\newtheorem{remark}[theorem]{Remark}
\newtheorem{definition}[theorem]{Definition}
\newtheorem{example}[theorem]{Example}

\newcommand{\fun}{\mathcal E}

\newcommand{\<}{\langle}
\renewcommand{\>}{\rangle}

\numberwithin{equation}{section}

\newcommand{\tn}{\textnormal}

\newcommand{\C}{$C^*$}
\newcommand{\AXB}{$\pre AX_B$}
\newcommand{\AXA}{$\pre AX_A$}
\newcommand{\AYB}{$\pre AY_B$}
\newcommand{\BYB}{$\pre BY_B$}
\newcommand{\BYC}{$\pre BY_C$}

\newcommand{\CZC}{$\pre CZ_C$}
\newcommand{\KK}{\mathcal K}
\newcommand{\LL}{\mathcal L}

\newcommand{\ot}{\otimes}
\newcommand{\ots}{\ot_A}
\newcommand{\otss}{\ot_B}
\newcommand{\otsk}{\ot_K}
\newcommand{\otsc}{\ot_C}
\newcommand{\otsi}{\ot_I}
\newcommand{\otsai}{\ot_{A/I}}

\newcommand{\ec}{\emph{C}}

\newcommand{\cat}{\ensuremath{\mathsf{C^*cor_{pair}}}}
\newcommand{\enchilada}{\ensuremath{\mathsf{C^*alg_{cor}}}}
\newcommand{\cats}{\ensuremath{\mathsf{C^*cor_{pair}^{reg}}}}

\newcommand{\ibm}{imprimitivity bimodule}

\newcommand{\otsbb}{\otimes_{B_M}}

\makeatletter
\let\original@footnotemark\footnotemark
\newcommand{\align@footnotemark}{%
  \ifmeasuring@
    \chardef\@tempfn=\value{footnote}%
    \original@footnotemark
    \setcounter{footnote}{\@tempfn}%
  \else
    \iffirstchoice@
      \original@footnotemark
    \fi
  \fi}
\pretocmd{\start@align}{\let\footnotemark\align@footnotemark}{}{}
\makeatother

\newcommand{\AMB}{$\pre AM_B$}
\newcommand{\ANB}{$\pre AN_B$}
\newcommand{\AAA}{$\pre AA_A$}

\newcommand{\BNC}{$\pre BN_C$}
\newcommand{\KKA}{$\pre KK_A$}
\newcommand{\CTK}{$\pre CT_K$}
\newcommand{\CNA}{$\pre CN_A$}
\newcommand{\KX}{$\pre K{KX}_K$}

\newcommand{\AMMB}{$\pre AM'_B$}
\newcommand{\UMM}{$U_{M'}$}
\newcommand{\otsmm}{\ot_{B/B_M}}

\newcommand{\AI}{$\pre A(A/I)_{A/I}$}
\newcommand{\MI}{$\pre {A/I}M_B$}
\newcommand{\NI}{$\pre {A/I}N_B$}
\newcommand{\BM}{$\pre B(B/B_M)_{B/B_M}$}
\newcommand{\YBM}{$\pre {B/B_M}(Y/YB_M)_{B/B_M}$}

\newcommand{\otsm}{\otimes_{B_M}}

\newcommand{\OX}{\mathcal{O}_X}
\newcommand{\OY}{\mathcal{O}_Y}
\newcommand{\OZ}{\mathcal{O}_Z}

\newcommand{\OXXI}{\mathcal{O}_{X/XI}}

\newcommand{\OIX}{\mathcal{O}_{IX}}

\newcommand{\otsp}{\otimes_{A/I}}

\newcommand{\oba}{\pre {A}}

\newcommand{\gae}{\lower 2pt \hbox{$\, \buildrel {\scriptstyle >}\over {\scriptstyle
\sim}\,$}}

\newcommand{\lae}{\lower 2pt \hbox{$\, \buildrel {\scriptstyle <}\over {\scriptstyle
\sim}\,$}}

\newcommand{\MU}[1]{
\setbox0\hbox{$#1$}
\setbox1\hbox{$W$}
\ifdim\wd0>\wd1 #1^{\sim} \else \widetilde{#1} \fi
}

\newcommand{\pre}[1]{{}_{#1}}

\begin{document}
\title[Exactness of the Cuntz-Pimsner Construction]{Exactness of the Cuntz-Pimsner Construction}

\author{Menev\c se Ery\"uzl\"u Paulovicks %\thanks{support}\\
		%Department of Mathematics\\
		%University of Iowa\\
		%Iowa City, IA  52242-1419\\
		%tomforde@math.uiowa.edu
}

\address{Department of Mathematics, University of Colorado Boulder, Boulder, CO 80309-0395}
\email{menevse.paulovicks@gmail.com}

%\address{Department of Mathematics \\ University of Colorado \\ Colorado Springs, CO 80918-3733 \\USA}
%\email{mtomford@uccs.edu}

%\thanks{This work was supported by a grant from the Simons Foundation (\#527708 to Mark Tomforde)}

%new Award Number: 527708

\date{\today}

\subjclass[2020]{Primary 46L08; Secondary 18B99}

\keywords{$C^*$-correspondence, Cuntz-Pimsner algebra, exact sequence, exact functor}

\begin{abstract}
In prior work we described how the Cuntz-Pimsner construction may be viewed as a functor. 
The domain of this functor is a category whose objects are \C-correspondences 
and morphisms are isomorphism classes of certain pairs comprised of a \C-correspondence and 
an isomorphism.   The codomain is the well-studied category whose objects are \C-algebras 
and morphisms are isomorphism classes of \C-correspondences. In this paper we show that certain fundamental results in the theory of Cuntz-Pimsner algebras are direct consequences of the functoriality of the Cuntz-Pimsner construction. In addition, we describe exact sequences in the target and domain categories, and prove that the Cuntz-Pimsner functor is exact.

%We construct exact sequences in a category where objects are  \C-correspondences and morphisms are isomorphism classes of pairs consist of a \C-correspondence and an isomorphism. We show that  the functor that maps \C-correspondences to Cuntz-Pimsner algebras is exact. As an applications we give a very quick and clear proof to some known results. 
\end{abstract}

\maketitle

\section{Introduction}
In  \cite{enchcor} we introduced a categorical framework for viewing the Cuntz-Pimsner construction as a functor, which thereby allows one to determine relationships between Cuntz-Pimsner algebras from relationships between the defining $C^*$-correspondences. The domain of this functor is the category \cat , introduced in \cite{enchcor}, that has \C-correspondences as objects, and a morphism from \AXA\ to \BYB\ is the isomorphism class of the pair (\AMB, $U_M$), where \AMB\ is a \C-correspondence satisfying certain conditions, and  \[U_M: \pre A(X\ots M)_B \rightarrow \pre A(M\otss Y)_B\] is a \C-correspondence isomorphism. The codomain category \enchilada , which has sometimes been called the Enchilada Category in the literature, has \C-algebras as objects and isomorphism classes of \C-correspondences as morphisms. For any pair (\AMB, $U_M$) one can construct a covariant representation ($\pi,\Phi$) of \AXA\ on \mbox{$\KK(M\otss \OY)$}. Then the universal property of Cuntz-Pimsner algebras assures the existence of an associated homomorphism \mbox{$\sigma_ {(\pi,\Phi)}: \OX\rightarrow \KK(M\otss\OY),$} which allows us to view $M\otss\OY$ as an $\OX$--$\OY$-correspondence.  It is shown in \cite{enchcor} that there exists a functor $\fun$ from \cat\ to \enchilada\ that maps a  \C-correspondence \AXA\ to its  Cuntz-Pimsner algebra $\OX$,  and a morphisms from \AXA $\rightarrow$ \BYB\ is mapped to the isomorphism class of  an $\OX$--$\OY$-correspondence. The first part of this paper is devoted to using this functor to obtain  some well-known Cuntz-Pimsner algebra  results. Specifically, if ($\Upsilon,t$) is a universal covariant representation of \AXA, and $I$ is a positive $X$-invariant ideal of $A$, then the following hold:

\begin{enumerate}[(i)]

\item $\mathcal{O}_{IX}$ is isomorphic to the smallest hereditary subalgebra of $\OX$ containing $\Upsilon(I)$.
\item $\mathcal{O}_{IX}$ is Morita equivalent to the ideal $\left\<\Upsilon(I)\right\>$ generated by $\Upsilon(I)$ in $\OX$.
\item If  \AXA\ is regular and $I$ is an $X$-invariant ideal then $\OX/\<\Upsilon(I)\>\cong \OXXI.$

\end{enumerate}

Items (i) and (ii) were proven by Katsura in \cite[Proposition~9.3 and Proposition~9.5]{katsura} by using what are called \emph{O}-pairs. Item (iii) can be shown by combining \cite[Lemma~6.3]{MT} and \cite[Theorem~3.1]{fowler}. Item (iii) can also be  deduced by combining \cite[Proposition~5.3]{katsura} and  \cite[Proposition~8.5]{katsura}; however, this requires a deep understanding of \emph{O}-pairs and the properties of \C-algebras generated by such pairs. The first objective of this paper is to show that the functor established in \cite{enchcor} can be applied to obtain easier and more direct proofs of these three results.  %We show that this construction does not only tell us the Morita equivalence for item 2, but it also gives us the imprimitivity bimodule between $\mathcal{O}_{IX}$  and $\<\Upsilon(I)\>$. 

The work that is presented in the second part of this paper was motivated by a question frequently asked by audience members when presenting the results of [3], namely: ``Can one define exact sequences in the domain and codomain categories so that the Cuntz-Pimsner functor $\fun$  is exact?''  In order to answer this question, one needs to know what is meant by an exact sequence in both the domain and codomain categories. One of the difficulties in answering this question is that one can not identify images in either categories, and it is suspected that images may not exist in either  categories.  Therefore the usual ``kernel-image definition'' of exactness may not apply. To circumvent this obstruction in each category, we define \mbox{$0 \rightarrow A \xrightarrow{f} B \xrightarrow{g} C \to 0$}  to be a short exact sequence when $f$ is a categorical kernel of $g$ and $g$ is a categorical cokernel of $f$ (see Definition~\ref{def-ench} and Definition~\ref{def-eccor}). We characterize this ``kernel-cokernel definition'' of exactness in each category, showing that kernel-cokernel exactness is a tractable concept in these categories.  We prove that with the kernel-cokernel definition of short exact sequence, the Cuntz-Pimsner functor is exact. We end the paper by showing that as an immediate consequence of exactness one can obtain the results (i),(ii), and (iii) listed above for the case when \AXA\ is regular.

\section{Preliminaries}

Let $\mathcal{C}$ be a pointed category (a category with a zero object), and let $f : A \rightarrow B$ be a morphism. A \emph{kernel} of $f$ is a pair $(K,k)$ consists of an object $K$ and a morphism  $k : K\rightarrow A$ such that

%consists of an object $K$ and a morphism  $k : K\rightarrow A$ such that
\begin{itemize}
\item $f\circ k=0$;
\item  whenever a morphism $h:D\rightarrow A$ satisfies $f\circ k=0$ there exists a unique morphism $p:D\rightarrow K$ such that $k\circ p =h$. 
\end{itemize}

A \emph{cokernel} of $f$ is a pair $(C,c)$ consists of an object $C$ and a morphism  $c: B\rightarrow C$ such that 
%consists of an object $C$ and a morphism $c: B\rightarrow C$ such that 
\begin{itemize}
\item $c\circ f =0;$
\item whenever a morphism $h: B\rightarrow D$ satisfies $h\circ f=0$  there exists a unique morphism $p:C\rightarrow D$ such that $p\circ c=h.$ 
\end{itemize}

 We say that $f$ is a \emph{monomorphism}  if for all morphisms $g: C \to A$ and $h : C \to A$ in $\mathcal{C}$, we have $f \circ g = f \circ h$ implies $g=h$.   We say that $f$ is an \emph{epimorphism} if for all morphisms $g: B \to C$ and $h : B \to C$ in $\mathcal{C}$, we have $g\circ f = h \circ f$ implies $g=h$.  
%When $m : I \to Y$ is an image of $f$ we refer to the composition $f = m \circ e$ as an \emph{image factorization} of $f$ in $\mathcal{C}$. 

A \C-correspondence \AXB\ is a right Hilbert $B$-module equipped with a left action given by a homomorphism $\varphi_X: A\rightarrow \LL(X),$ where $\LL(X)$ denotes the \C-algebra of adjointable operators on $X$. We denote the kernel of the left action homomorphism $\varphi_X$ by $\ker\varphi_X$. For a \C-correspondence \AXB\ define \mbox{$A\cdot X = \{\varphi_X(a)x: a\in A, x\in X \}$.}  The correspondence \AXB\ is called \emph{nondegenerate} if $A\cdot X = X$. In this paper \emph{all our correspondences will be nondegenerate by standing hypothesis.} 
A \C-correspondence \AXB\ is called \emph{injective} if the left action $\varphi_X: A\rightarrow \LL(X)$ is injective; it is called \emph{proper} if  $\varphi_{X}(A)$ is contained in the \C-algebra $\KK(X)$ of compact operators on $X$. A \C-correspondence \AXB\ is called \emph{regular} if it is both injective and proper.  For a \C-correspondence \AXB\ we denote the closed span of $B$-valued inner products $\<X,X\>_B$ by $B_X$.  One of the \C-correspondence properties we use frequently in this paper is the following: let $I$ be an ideal of $B$ such that $B_X\subset I$. Then, $X$ can be viewed as an $A$--$I$-correspondence \cite[Lemma~3.2]{ench}.

A \C-correspondence isomorphism from \AXB\ to \AYB\  is a bijective linear map $\Phi: X\rightarrow Y$  satisfying
\begin{enumerate}[(i)]
\item $\Phi(a\cdot x)=a \cdot \Phi(x)$,
\item $ \<x,z\>_B = \<\Phi(x), \Phi(z)\>_{B}$,
\end{enumerate}
for all $a\in A$, and $x,z\in X.$ We let $\ad\Phi: \LL(X)\rightarrow \LL(Y)$ denote the associated \C-algebra isomorphism defined by $\ad\Phi(T)=\Phi\circ T\circ \Phi^{-1}.$

%\begin{definition}\label{precor} An $A-B$ bimodule $X_0$ is called a \emph{pre-correspondence} if it has a $B$-valued semi-inner product satisfying 
%\[ \<x, y\cdot b\> = \<x,y\> b ,   \midtext{} \<x,y\>^*=\<y,x\> \]
%and $\<a\cdot x, a\cdot x\> \leq \norm{a}^2\<x,x\>$ for all $a\in A, b\in B$ and $x,y\in X_0$. Modding out by the elements of length $0$ and completing gives a \C-correspondence \AXB . We call \AXB\ the \emph{completion} of the pre-correspondence $X_0$. 
%\end{definition}

%\begin{proposition}\label{pre}\tn{\citep[Lemma 1.23]{enchilada}}
%Let $X_0$ be an $A-B$ pre-correspondence given with the completion \AXB , and  let  $Z$ be an $A-B$ correspondence. If there is a map $\Phi : X_0 \rightarrow Z$ satisfying 
%\[ \Phi (a\cdot x) = \varphi_Z (a) \Phi(x)  \midtext{and} \<\Phi(x), \Phi(y) \>_B = \<x, y\>_{B} , \]
%for all $a\in A$ and $x,y\in X_0$, then $\Phi$ extends uniquely to an injective $A-B$ correspondence homomorphism $\tilde{\Phi}: X \rightarrow Z . $
%\end{proposition}

The \emph{balanced tensor product} $X\otimes_BY$ of
an $A-B$ correspondence $X$ and a $B-C$ correspondence $Y$ is
formed as follows:
the algebraic tensor product $X\odot Y$
is a pre-correspondence with the $A-C$ bimodule structure satisfying
\[
a(x\otimes y)c=ax\otimes yc
\righttext{for}a\in A,x\in X,y\in Y,c\in C,
\]
and the unique $C$-valued semi-inner product whose values on elementary tensors are given by
\[
\<x\otimes y,u\otimes v\>_C=\<y,\<x,u\>_B\cdot v\>_C
\righttext{for}x,u\in X,y,v\in Y.
\]

This semi-inner product defines a $C$-valued inner product on the quotient $X{\odot}_BY$ of $X\odot Y$ by the subspace generated by elements of form 
\[ x\cdot b \otimes y - x\otimes \varphi_Y(b)y \righttext{($x\in X$, $y\in Y$, $b\in B$)}.\]
The completion $X\otimes_B Y$ of  $X{\odot}_BY$ with respect to the norm coming from the $C$-valued inner product is an $A-B$ correspondence, where the left action is given by 
 \[A\rightarrow \LL(X\otss Y),  \righttext{$a\mapsto \varphi_X(a)\ot 1_Y,$}\]
for $a\in A.$ We denote the canonical image of $x\otimes y$ in $X\otss Y$ by $x\otss y$. %The term \emph{balanced} refers to the property
%\[
%x\cdot b\otss y=x\otss b\cdot y
%\righttext{for}x\in X,b\in B,y\in Y,
%\]
%which is a consequence of the construction. 

\begin{proposition}\tn{\cite[Proposition~3.1]{ench}}\label{zero} For \C-correspondences \AXB\ and \BYC\ we have 
\[ \pre A(X\otss Y)_B \cong \pre A0_B \iff \pre A(X\otss Y)_B=\pre A0_B \iff  B_X\subset \ker\varphi_Y. \]
\end{proposition}

\begin{lemma}[\cite{fowler}]\label{fowler} Let $X$ be a \C-correspondence over $A$ and let $\pre AY_B$ be an injective \C-correspondence. Then the map 
$\iota: T \mapsto T\ot 1_Y$  gives an isometric homomorphism of $\LL(X)$ into $\LL(X\ots Y).$ If, in addition, $\varphi_Y(A)\subset\KK(Y)$, then $\iota$ embeds $\KK(X)$ into $\KK(X\ots Y).$ 
\end{lemma}

A \emph{Hilbert bimodule} \AXB\ is a \C-correspondence
that is also equipped with an $A$-valued inner product $\pre A\<\cdot,\cdot\>$,
which satisfies
\[
\pre A\<a\cdot x,y\>=a\cdot \pre A\<x,y\>\midtext{and}\pre A\<x,y\>^*=\pre A\<y,x\>
\]
for all $a\in A,x,y\in X$,
as well as the \emph{compatibility property}
\[
\pre A\<x,y\>\cdot z=x\cdot \<y,z\>_B
\righttext{for}x,y,z\in X.
\]
A Hilbert bimodule $\pre AX_B$ is \emph{left-full} if the closed span of $\pre A\<X,X\>$ is all of $A$. %In any event, if $X$ is an $A-B$ \hbm\ then
%$A_X$ is an ideal of $A$ that is mapped isomorphically onto $\KK(X)$ via $\phi_X$.

An \emph{\ibm} \AXB\ is a Hilbert bimodule  that is full on both the left and the right. The \emph{identity correspondence} on $A$ is the Hilbert bimodule $\pre AA_A$ where the bimodule structure is  given by multiplication, and the inner products are given by 
\[ \oba\<a,b\>=ab^*, \hspace{.5cm} \<a,b\>_A= a^*b, \righttext{for $a,b\in A$}.\]

A \emph{representation} $(\pi,t)$ of a \C-correspondence \AXA\ on a \C-algebra $B$ consists of a $*-$homomorphism $\pi: A\rightarrow B$ and a linear map $t: X\rightarrow B $ such that 
\[
\pi(a)t(x)=t(\varphi_X(a)(x)) \midtext{and} t(x)^* t(y)=\pi(\<x,y\>_A), 
\]
for $a\in A$ and $x, y\in X$, where $\varphi_X$ is the left action homomorphism associated with \AXA.  For any representation $(\pi,t)$ of \AXA\ on $B$, there is an associated homomorphism $\psi_t:\KK(X)\rightarrow B$ satisfying $\psi_t(\theta_{x,x'})=t(x)t(x')^*$ for $x,x'\in X$. The representation $(\pi,t)$ is called \emph{injective} if $\pi$ is injective, in which case $t$ is an isometry. We denote the \C-algebra generated by the images of $\pi$ and $t$ in $B$ by \C$(\pi, t).$

Consider a \C-correspondence \AXA . The ideal $J_X$ is define as
\begin{align*}
 J_X &= \varphi_X^{-1}(\KK(X))\cap (\ker\varphi_X)^{\perp}\\
 &=\text{\{$a\in A$ : $\varphi_X(a)\in\KK(X)$ and $ab=0$ for all $b\in\ker\varphi_X$}\},
\end{align*}
and is called the \emph{Katsura ideal}. Notice here that for a regular  \AXA\ we have  $J_X=A.$

A representation $(\pi, t)$ of \AXA\ is called \emph{covariant} if $\pi(a)=\Psi_{t} (\varphi_X(a) )$, for all $a\in J_X .$  The \C-algebra generated by the universal covariant representation  of \AXA\ is called the \emph{Cuntz-Pimsner} algebra $\OX$ of \AXA .

\section{Categories and the covariant representation}

In this section we briefly explain the construction of the functor $\fun$ defined in \cite{enchcor}, and recall the related categories.  The range category \tn{\enchilada} of  $\fun$ is sometimes called  ``the enchilada category'' as in \cite{ench}. In this category our objects are $C^*$-algebras, and a morphism from $A$ to $B$ is the isomorphism class of  an $A$--$B$-correspondence. The composition  [\BYC]$\circ$[\AXB] is the isomorphism class of the balanced tensor product $\pre A(X\otss Y)_C$; the identity morphism on $A$ is the isomorphism class of the identity correspondence $\pre AA_A$, and the zero morphism $A\rightarrow B$ is $[\pre A0_B].$  Note that a morphism [\AXB] is an isomorphism in \enchilada\ if and only if \AXB\ is  an \ibm\ \cite[Proposition~2.6]{nat}. %A detailed study of the category \enchilada\ can be found in \cite{taco}. 

%\begin{itemize}%[\normalfont(1)]
%\item[(1)] A kernel of the morphism  \tn{[\AXB]} is the pair $\left(K_X,[\pre {K_X}(K_X)_A]\right)$.
%\item[(2)] A  cokernel of the morphism \tn{[\AXB]} is the pair $\left(B/B_X, [\pre B(B/B_X)_{B/B_X}]\right)$.
%\item[(3)] Schubert Image of \tn{[\AXB]} is $[\pre {B_X}{B_X}_B].$
%\end{itemize}
We need the following definition for the domain category. 

\begin{definition}\label{morphiso}\cite[Definition~3.1]{enchcor} For \C-correspondences  \AXA , \BYB , and $A$--$B$-correspondences \AMB, \ANB , let $U_M: X\ots M \rightarrow M\otss Y$ and $U_N: X\ots N \rightarrow N\otss Y$ be  $A$--$B$-correspondence isomorphisms. The pairs (\AMB , $U_M$) and (\ANB, $U_N$) are called \emph{isomorphic} if 
% The isomorphism class of the pair (\AMB , $U_M$), denoted by $[\pre AM_B , U_M]$, consists of pairs (\ANB, $U_N$) such that\
\begin{itemize}
\item there exists an isomorphism  $\xi$: \AMB\ $\rightarrow$ \ANB; and 
\item the diagram
\[
\begin{tikzcd}
X\ots M \arrow{r}{1\ot \xi} \arrow[swap]{d}{U_M} & X\ots N \arrow{d}{U_N} \\
M\otss Y \arrow{r}\arrow{r}{\xi\ot 1_Y} & N\otss Y
\end{tikzcd}
\]
commutes.

\end{itemize}

We denote the isomorphism class of the pair (\AMB, $U_M$) by $[\pre AM_B , U_M]$.

\end{definition}

\begin{remark}\label{notation} For a \C-correspondence \AMB , let $I$ and $J$ be ideals of $A$ and $B$, respectively. We denote the map 
\[M\otss J \rightarrow MJ, \righttext{ $m\otss j \mapsto m\cdot j$}\]
by $\xi_{(r,M,J)}$, where $m\in M$,  $j\in J$. This map defines an $A$--$B$-correspondence isomorphism as well as  an $A$--$J$-correspondence isomorphism. Similarly, we denote 
the map \[I\ots M\rightarrow IM, \righttext{ $i\ots m \mapsto i\cdot m$}\] by $\xi_{(l,M,I)},$ where $m\in M$, $i\in I.$ This map defines  an $A$--$B$-correspondence isomorphism as well as an $I$--$B$-correspondence isomorphism. 
\end{remark}

\begin{theorem}[{\cite[Theorem~3.2]{enchcor}}]
There exists a category \cat\ such that 
\begin{itemize}
\item objects are \C-correspondences \AXA ;
\item morphisms \AXA $\rightarrow$ \BYB\ are isomorphism classes  \tn{[\AMB , $U_M$]}  where $U_M$ denotes an $A$--$B$-correspondence isomorphism $X\ots M \rightarrow M\otss Y$, and \AMB\ is a proper correspondence satisfying $J_X\cdot M \subset M\cdot J_Y$;
\item the composition \tn{[\BNC, $U_N$]$\circ$[\AMB , $U_M$]} is given by the isomorphism class \[\tn{[$\pre A(M\otss N)_C$, $U_{M\otss N}$]} \] 
where $U_{M\otss N}$ denotes the isomorphism $(1_M\ot U_N)(U_M\ot 1_N)$;
\item the identity morphism on \AXA\ is \tn{[\AAA , $U_A$]}, where $U_A$ denotes the isomorphism $\xi_{l,X,A}^{-1}\circ\xi_{r,X,A}: X\ots A \rightarrow A\ots X$.
\end{itemize}
\end{theorem}

Let [\AMB, $U_M$]: \AXA $\rightarrow$ \BYB\  be a morphism in \cat . Denote the universal covariant representation of \BYB\ by $(\Upsilon, t).$  Let  $V_Y: Y\otss \OY \rightarrow \overline{t(Y)\OY}$ be the isomorphism determined on elementary tensors by
\[ V_Y(y\otss S)=t(y)S\]
for $y\in Y$, $S\in\OY$. Define  $T: X\rightarrow \LL(M,M\otss Y)$  by 
\[ T(x)(m)= U_M(x\ots m),\]
for $x\in X$, $m\in M$. Next, define a linear map $\Phi: X\rightarrow \KK(M\otss \OY)$ by 
\[ \Phi(x) = (1_M\ot V_Y)(T(x)\ot 1_Y).\]  and a homomorphism $\pi: A\rightarrow \KK(M\otss\OY)$ by 
\[ \pi(a)= \varphi_M\ot 1_{\OY}.\] The pair $(\pi, \Phi)$ is a covariant  representation of \AXA\ on $\KK(M\otss\OY)$ \mbox{\cite[Proposition~4.2]{enchcor}}, and it is called  the  \ec-\emph{covariant representation} of \AXA .  It is injective when the homomorphism $\varphi_M$ is. By the universal property of $\OX$ we obtain a $*$-homomorphism 
$\sigma_{(\pi,\Phi)}: \OX\rightarrow \KK(M\otss \OY)$, which provides a left action of $\OX$ on the Hilbert \mbox{$\OY$-module}  $M\otss \OY$ and allows us to view $M\otss\OY$ as a proper $\OX$--$\OY$-correspondence. It is important to note that  the \ec-covariant representation $(\pi, \Phi)$ admits a gauge action. Consequently, the homomorphism $\sigma_{(\pi,\Phi)}$ is an isomorphism onto \C$(\pi, \Phi)$ when \AMB\ is an injective $C^*$-correspondence \cite[Theorem~4.13]{enchcor}.

\begin{theorem}[{\cite[Theorem~5.1]{enchcor}}]

Let \tn{[\AMB, $U_M$]:} \AXA $\rightarrow$ \BYB\  be a morphism in \cat . Then the assignments \AXA $\mapsto \OX$ on objects and \[  [\pre AM_B, U_M] \mapsto [\pre {\OX}(M\otss \OY)_{\OY}] \]
on morphisms define a functor $\fun$ from \cat\ to \enchilada .
\end{theorem}

Next Proposition is crucial for this paper.

\begin{proposition}\label{kerneliso}
Let $[\pre AM_B, U_M]:$ \AXA $\rightarrow$ \BYB\ be a morphism in \cat , where \AXA\ is a regular correspondence,  and let  $\sigma: \OX\rightarrow\KK(M\otss\OY)$ be the associated homomorphism. Denote  the universal covariant representation of \AXA\ by $(\Upsilon,t)$. Then $\ker\sigma$ is the ideal $\<\Upsilon(\ker\varphi_M)\>$  generated by $\Upsilon(\ker\varphi_M)$ in $\OX.$ 
\end{proposition}

\begin{proof}
It suffices to show the equality $\ker\sigma\cap \Upsilon(A) = \<\Upsilon(\ker\varphi_M)\>\cap\Upsilon(A)$, since gauge invariant ideals of $\OX$ are distinguished by their intersection with $\Upsilon(A)$  when \AXA\ is regular \cite[Corollory~8.7]{katsura}. One can easily verify  that $\<\Upsilon(\ker\varphi_M)\>\subset \ker\sigma$. Let  $\Upsilon(a)\in\ker\sigma$. Then we have 
\[ 0=\sigma(\Upsilon(a))=\varphi_M(a)\ot 1_{\OY}.\]
This implies  $\varphi_M(a)=0$ by Lemma~\ref{fowler}. And thus, $a\in\ker\varphi_M$, which means $\Upsilon(a)\in \Upsilon(A)\cap\<\Upsilon(\ker\varphi_M)\>.$
\end{proof}

\section{Invariant Ideals and Structure Theorems}

\begin{definition} Let \AXA\ be a \C-correspondence. For an ideal $I$ of $A$, define an ideal $X^{-1}(I)$ of $A$ by 
\begin{align*}
X^{-1}(I)&= \{ a\in A: \<x,a\cdot y\>_A\in I   \text{ for all }  x,y\in X\}.
\end{align*}
An ideal $I$ of $A$ is said to be \emph{positive $X$-invariant} if $IX\subset XI$, \emph{negative $X$-invariant} if $J_X\cap X^{-1}(I)\subset I,$ and \emph{$X$-invariant} if $I$ is both positive and negative invariant.
\end{definition}

Note that $I$ is a positive $X$-invariant ideal of $A$ if and only if $\<X, IX\>_A\subset I.$  When that's the case, we have $IX=IX\<IX,IX\>_A\subset IXI$. Therefore, we have the equality \mbox{$IX=IXI$}. Consequently, the $I$--$A$-correspondence $IX$ can be viewed as a \C-correspondence over $I$.

\begin{lemma}\label{idealmorph} Let \AXA\ be a $C^*$-correspondence and I be a positive X-invariant ideal of A. Denote the  $I$--$A$-correspondence isomorphism $\xi_{(l,X,I)}^{-1}\circ \xi_{(r,IX,I)}: IX\otsi I \rightarrow I\ots X$ by $U_I$, where $\xi_{(l,X,I)}$ and  $\xi_{(r,IX,I)}$ are the $I$--$A$-correspondence isomorphisms defined as in Remark~\ref{notation}. Then, the isomorphism class $[\pre II_A, U_I]: \pre IIX_I\rightarrow$ \AXA\ is a morphism in \cat . 
\end{lemma}

\begin{proof}
It suffices to show $J_{IX}\cdot I \subset J_{X}$, which follows immediately from the fact that $J_{IX}=I\cap J_X$ \citep[Proposition~9.2]{katsura}.
\end{proof}

\begin{lemma}\label{compact}  For C*-algebras  $A$ and $B$, let  $A\subset B$. Then we have the \C-algebra isomorphism $\KK(AB)\cong ABA$, where $AB$ is viewed as a Hilbert $B$-module .\end{lemma}

\begin{proof}
For any $x\in ABA$, consider the operator $T_x: AB \rightarrow AB$ defined by \mbox{$T_x(y)=xy$}, where $y\in AB.$ Then each $T_x$ is an element of $\KK(AB)$, and the map \mbox{$L: ABA\rightarrow \KK(AB)$} defined by $x\mapsto T_x$ is an injective $*$-homomorphism. Now take any $\theta_{a_1b_1,a_2b_2}\in \KK(AB)$.  We have $\theta_{a_1b_1,a_2b_2}=T_{a_1b_1b_2^*a_2^*}=L(a_1b_1b_2^*a_2^*)$. And thus, $L$ is surjective. \end{proof}

\begin{theorem}\label{posinv} Let \AXA\ be a $C^*$-correspondence and I be a positive X-invariant ideal of $A$. Let $(\Upsilon, t)$ be the universal covariant representation of \AXA. Then we have the following:
\begin{enumerate}[\normalfont(1)]
\item $\fun\left([\pre II_A, U_I]\right)=\left[ \pre {\mathcal{O}_{IX}}(I\ots \OX)_{\OX} \right]$ is an isomorphism class of a left-full Hilbert bimodule. 
\item $\mathcal{O}_{IX}$ is isomorphic to the smallest hereditary subalgebra of $\OX$ containing $\Upsilon(I)$ \tn{\cite[Proposition~9.3]{katsura}}.
\item $\mathcal{O}_{IX}$ is Morita equivalent to the ideal $\left\<\Upsilon(I)\right\>$ generated by $\Upsilon(I)$ in $\OX$. \tn{\cite[Proposition~9.5]{katsura}}.
\end{enumerate}
\end{theorem} 

\begin{proof} Let $ \xi: I\ots\OX\rightarrow \Upsilon(I)\OX$ denote the Hilbert $\OX$-module isomorphism defined on elementary tensors by $i\ots S\mapsto\Upsilon(i)S$.
Denote the \emph{C}-covariant representation of \AXA\ by ($\pi, \Phi$), and let $L: \Upsilon(I)\OX\Upsilon(I)\rightarrow \KK(\Upsilon(I)\OX)$ be the isomorphism defined as in the proof of Lemma~\ref{compact}. Then we have the following diagram.

\begin{center}
%\begin{minipage}{0.45\textwidth}
\begin{tikzpicture}[node distance=1.6 cm, scale=1, transform shape]
\node (K)  [scale=0.7]  {$\KK(I\ots\OX)$};
\node (IX) [above of=K, left of=K, scale=0.7] {$IX$};
\node (I) [below of=K, left of=K, scale=0.7] {$I$};
\draw[->] (IX) to node [scale=0.6, right]{$\Phi$} (K);
\draw[->] (I) to  node [scale=0.6, right]{$\pi$}(K);
%\node (K) [left of=B, scale=0.7] {$\KK(X)$};
\node (L) [right of=K, right of=K, scale=0.7]{$\KK(\Upsilon(I)\OX)$};
\draw[->] (K) to node[scale=0.6,above]  {$\ad\xi$} (L);
\node (LL) [right of=L, xshift=2cm, scale=0.7]{$\Upsilon(I)\OX\Upsilon(I)$};
\draw[->] (L) to node[scale=0.6,above] {$L^{-1}$}  (LL);
%\draw[->] (X) to node[scale=0.6] {$\Phi_1$} (L);
%\draw[->] (A) to node[swap, scale=0.6] {$\pi_1$} (L);
%\node (LLL) [right of=LL, right of=LL, xshift=1.5cm, scale=0.7]{$\KK(M\otss N\otsc \OZ)$}; 
%\draw[->] (LL) to node[scale=0.6]{$\ad U$} (LLL);
%\draw[->] (X) to node[scale=0.6] {$\Phi_2$} (LLL);
%\draw[->] (A) to node[scale=0.6, swap] {$\pi_2$} (LLL);
%\draw[->, out=10, in=40] 
\end{tikzpicture}
%\end{minipage}
\end{center}

We claim that $\ad\xi^{-1}\circ L$ is an isomorphism onto \C($\pi,\Phi$). It suffices to show the equalities 
\[\ad\xi[\Phi(ixj)]=L\left(\Upsilon(i)t(x)\Upsilon(j)\right) \text{ and } \ad\xi[\pi(iaj)]=L\left(\Upsilon(iaj)\right),\]
for any $i,j\in I, x\in X,$ and $a\in A.$ Let $V: X\ots\OX\rightarrow \overline{t(X)\OX}$ denote the $A-\OX$ correspondence isomorphism defined on elementary tensors by $x\ots S\mapsto t(x)S$, for any $x\in X$, $S\in\OX$.  For $k\in I, S\in \OX$ we have 
\begin{align*}
\xi\Phi(ixj)(k\ots S) &= \xi(1_I\ot V)U_I(ixj\otsi k)\ots S\\
 &= \xi(a\ots t(z)S) & \text{(where $a\in I, z\in X$ with $az=ixjk$)}\\
&= \Upsilon(a)t(z)S\\
&=t(ixjk)S.
\end{align*}
On the other hand, we have 
\[ L\left(\Upsilon(i)t(x)\Upsilon(j)\right)\xi(k\ots S)= \Upsilon(i)t(x)\Upsilon(j)\Upsilon(k)S=t(ixjk)S,\]
which proves the first equality. For the second equality we observe that 
\[\xi\pi(iaj)(k\ots S)=\xi(iajk\ots S)=\Upsilon(iajk)S=L\left(\Upsilon(iaj)\right)\xi(k\ots S),\] which proves our claim.

 We may now conclude that the injective $*$-homomorphism $\sigma: \OIX\rightarrow \KK(I\ots \OX)$ is onto. And thus, the \C-correspondence  $\pre {\mathcal{O}_{IX}}(I\ots \OX)_{\OX}$  is a left-full Hilbert bimodule, which implies $\OIX$ and $\<I\ots\OX, I\ots\OX\>_{\OX}=\<\Upsilon(I)\>$ are Morita equivalent \C-algebras. Moreover, by Lemma~\ref{compact} we have $\OIX\cong  \KK(I\ots \OX)\cong \KK(\Upsilon(I)\OX)\cong\Upsilon(I)\OX\Upsilon(I)$, which proves item (2). 
\end{proof}

\begin{remark}\label{quotient}  Let \AXA\ be a  \C-correspondence, and let $I$ be a positive $X$-invariant ideal of $A$. Let $p: A\rightarrow A/I$ and $q: X\rightarrow X/XI$ be the natural quotient maps. Then, $X/XI$ can be viewed as a \C-correspondence over $A/I$ with the module actions and the inner product are given by 
\[p(a)\cdot q(x)\cdot q(a')=q(axa') \hspace{1cm} \<q(x), q(y)\> = p\left(\<x,y\>_A\right),\]
for $a,a'\in A$, and $x,y\in X$. 

Assume \AXA\ is regular and $I$ is an $X$-invariant ideal. Then $X/XI$ is a regular correspondence as well: properness  of $X/XI$ is straightforward by construction. To see injectivity let $a\in A$, and  let $p(a)q(x)=0$ for all $x\in X$. Then $ax\in XI$ for all $x\in X$, which means $a\in X^{-1}(I)$. Since $I$ is $X$-invariant and \AXA\ is regular we have  $X^{-1}(I)\subset I$, and thus $p(a)=0$. 

Now, for a regular correspondence \AXA\ and an $X$-invariant ideal $I$, consider the isomorphisms
\begin{align*}
&i_1: X\ots A/I \rightarrow X/XI,  \hspace{1cm} x\ots p(a)\mapsto q(xa)\\
& i_2: A/I{\otimes}_{A/I} X/XI\rightarrow X/XI, \hspace{1cm} p(a)\otsp q(x)\mapsto q(ax),
\end{align*}
where $x\in X, a\in A. $ Then, [$A/I, U_{A/I}$]: \AXA $\rightarrow \pre {A/I}(X/XI)_{A/I}$ is a morphism in \cat, where $U_{A/I}:=i_2^{-1}\circ i_1$. 

Note that for any $x\in X,$ $a\in A$, we have $U_{A/I}(x\ots p(a))= p(a')\otsp q(x'),$ for some $a'\in A, x'\in X,$ satisfying  $p(a')q(x')=q(x)p(a)$: $i_1(x\ots p(a))=q(xa)$. Since $q(xa)$ is an element of the non-degenerate correspondence $X/XI$, there exists $p(a')\in A/I, q(x')\in X/XI$ such that  $p(a')q(x')=q(x)p(a)$. 
\end{remark}

\begin{theorem}\label{quo}
Let \AXA\ be a regular correspondence and let $I$ be an $X$-invariant ideal. Then, we have the isomorphism $\OX/\<\Upsilon(I)\>\cong \mathcal{O}_{X/XI}$.
\end{theorem}

\begin{proof} Let ($\Upsilon, t$) and ($\tilde{\Upsilon}, \tilde{t}$) be universal covariant representations of  $X$ and $X/XI$, respectively. And, let $p: A\rightarrow A/I$ and $q: X\rightarrow X/XI$ be the  quotient maps. The map $\xi: A/I \otsp  \mathcal{O}_{X/XI}\rightarrow  \mathcal{O}_{X/XI}$ defined on elementary tensors by \mbox{$p(a)\ot S \mapsto \tilde{\Upsilon}(p(a))S$} is a Hilbert-$\mathcal{O}_{X/XI}$ module isomorphism, and extends to a \C-algebra isomorphism $\ad\xi: \KK(A/I \otsp  \mathcal{O}_{X/XI})\rightarrow  \KK(\mathcal{O}_{X/XI})$. Let $L: \mathcal{O}_{X/XI}\rightarrow \KK(\mathcal{O}_{X/XI})$ be the \C-algebra isomorphism defined by $L(S)T=ST$ for $S,T\in  \mathcal{O}_{X/XI}$. Denote the \ec-covariant representation of \AXA\ on $\KK(A/I\otsp \mathcal{O}_{X/XI})$ by $(\pi, \Phi).$ Then we have the following diagram.

\begin{center}
%\begin{minipage}{0.45\textwidth}
\begin{tikzpicture}[node distance=1.4 cm, scale=1, transform shape]
\node (K)  [scale=0.7]  {$\KK(A/I\otsp \mathcal{O}_{X/XI})$};
\node (X) [above of=K, left of=K, scale=0.7] {$X$};
\node (A) [below of=K, left of=K, scale=0.7] {$A$};
\draw[->] (X) to node [scale=0.6, right]{$\Phi$} (K);
\draw[->] (A) to  node [scale=0.6, right]{$\pi$}(K);
%\node (K) [left of=B, scale=0.7] {$\KK(X)$};
\node (L) [right of=K, right of=K, scale=0.7]{$\KK( \mathcal{O}_{X/XI})$};
\draw[->] (K) to node[scale=0.6,above]  {$\ad\xi$} (L);
\node (LL) [right of=L, xshift=0.7cm, scale=0.7]{ $\mathcal{O}_{X/XI}$};
\draw[->] (L) to node[scale=0.6,above] {$L^{-1}$}  (LL);
%\draw[->] (X) to node[scale=0.6] {$\Phi_1$} (L);
%\draw[->] (A) to node[swap, scale=0.6] {$\pi_1$} (L);
%\node (LLL) [right of=LL, right of=LL, xshift=1.5cm, scale=0.7]{$\KK(M\otss N\otsc \OZ)$}; 
%\draw[->] (LL) to node[scale=0.6]{$\ad U$} (LLL);
%\draw[->] (X) to node[scale=0.6] {$\Phi_2$} (LLL);
%\draw[->] (A) to node[scale=0.6, swap] {$\pi_2$} (LLL);
%\draw[->, out=10, in=40] 
\end{tikzpicture}
%\end{minipage}
\end{center}

We claim $\ad\xi^{-1}\circ L$ is an isomorphism onto \C$(\pi, \Phi)$. To prove our claim we first show $\tilde{t}(q(x))=L^{-1}\ad\xi(\Phi(x)),$ for $x\in X$. Let $a\in A, S\in  \mathcal{O}_{X/XI}.$ On one hand we have 
\[ L(\tilde{t}(q(x))\xi(p(a)\otsp S)= L(\tilde{t}(q(x))\tilde{\Upsilon}(p(a))S= \tilde{t}(q(x))\tilde{\Upsilon}(p(a))S= \tilde{t}(q(xa))S.\] Now, let $V: X/XI\otsp \mathcal{O}_{X/XI}\rightarrow \mathcal{O}_{X/XI}$ denote the  isomorphism defined on elementary tensors by $q(z)\otsp T\mapsto \tilde{t}(q(z))T$, where $z\in X$ and $T\in \mathcal{O}_{X/XI}$.  We have 
\begin{align*}
\xi\Phi(x)(p(a)\otsp S)&=\xi(1_{A/I}\ot V)(U_{A/I}(x\ots p(a))\otsp S)\\
&= \xi(1_{A/I}\ot V)(p(a')\otsp q(x')\otsp S) & \text{ where $q(x)p(a)=p(a')q(x')$\footnotemark} \\
&=\xi\left[p(a')\otsp \tilde{t}(q(x'))S\right]\\
&= \tilde{t}(q(xa))S. 
\end{align*}
\footnotetext{See the last paragraph of Remark~\ref{quotient}.}
It is easy to show $\tilde{\Upsilon}(p(a))=L^{-1}\ad\xi(\pi(a)),$ for any $a\in A$, completing the proof of our claim. We may now conclude that  $\sigma: \OX\rightarrow \KK(A/I\otsp \mathcal{O}_{X/XI})$ is surjective. Then by the first isomorphism theorem we have $\OX/\ker\sigma \cong \sigma(\OX)$. By using Proposition~\ref{kerneliso} we obtain $\OX/\<\Upsilon(I)\> \cong \KK(A/I\otsp \mathcal{O}_{X/XI})\cong  \mathcal{O}_{X/XI}$.\end{proof}

We next give a factorization property in \cat, which allows us to generalize the first item of Theorem~\ref{posinv}. But first we need a Lemma. 

\begin{lemma}\label{inv} For an  $A$--$B$-correspondence isomorphism $U_M: X\ots M \rightarrow M\otss Y$ we have the following.
\begin{enumerate}[\normalfont(1)]
\item The ideal $B_M=\overline{\<M,M\>_B}$ of $B$ is positive $Y$-invariant. 
\item$\ker\varphi_M $  is a positive $X$-invariant ideal of $A$. If \AXA\ and \BYB\ are regular correspondences, then $\ker\varphi_M$ is an $X$-invariant ideal.
\end{enumerate}
\end{lemma}

\begin{proof} For the first item we compute
\[ \<Y, B_M\cdot Y\>_B=\<M\otss Y, M\otss Y\>_B =  \<X\ots M, X\ots M\>_B = \< M, A_X\cdot M\>_B\subset B_M, \]
as desired.  For the second item denote $\ker\varphi_M$ by $K$. We have 
\[ 0=\<K\cdot M\otss Y, M\otss Y\>_B = \<K\cdot X\ots M, X\ots M\>_B=\<M, \<KX,X\>_A\cdot M\>_B,\]
which implies $\<KX, X\>_A\subset K$, as desired. Now, assume \AXA\ and \BYB\ are regular correspondences. Let $a\in X^{-1}(K)$. Then, $\<ax,x'\>_A\in K$ for any $x,x'\in X$. This means
\[ \<ax\ots m, x'\ots n\>_B =\<m, \<ax,x'\>_A\cdot n\>_B=0,\]
for any $x,x'\in X$ and $m,n\in M$. This implies  $a\in \ker\varphi_{X\ots M}=\ker\varphi_{M\otss Y}.$ Then for any $m,n\in M$ and  $y,y'\in Y$,  we have 
\[ 0=\<a\cdot m\otss y , m'\otss y'\>_B = \<y, \<am,m'\>_B\cdot y'\>_B,\]
which implies $\<am,m'\>_B\in \ker\varphi_Y$. Since \BYB\ is a regular correspondence, we conclude that $a\in\ker\varphi_M$. \end{proof}

Let \tn{[\AMB, $U_M$]: \AXA $\rightarrow$ \BYB} be a morphism in \cat. The first item of Lemma~\ref{inv} and Lemma~\ref{idealmorph} together imply that \mbox{\tn{[$\pre {B_M}(B_M)_B, U_{B_M}$]}: $\pre {B_M}(B_MY)_{B_M} \rightarrow$ \BYB} is a morphism in \cat. 

\begin{proposition}\label{factor}
For any morphism \tn{[\AMB, $U_M$]: \AXA $\rightarrow$ \BYB} in \cat, there exists a morphism \tn{$[\pre AM'_{B_M}, U_{M'}]$: \AXA$\rightarrow {\pre {B_M}(B_M Y)}_{B_M}$} such that the equality \[\tn{[$\pre AM_B, U_M$]=[$\pre {B_M}(B_M)_B, U_{B_M}$]$\circ$ [$\pre AM'_{B_M}, U_{M'}$]}\] holds. 
\end{proposition}

\begin{proof} Let  $\pre AM'_{B_M}$ be  the Hilbert $B$-module $M$ viewed as $A$--$B_M$-correspondence. Consider the following \C-correspondence isomorphisms:
\begin{align*}
& \iota: \pre A(M\otss B_M)_{B_M} \rightarrow \pre AM'_{B_M}, &m\otss b\mapsto m\cdot b\\
& l: \pre A(M'\otsbb B_M)_B\rightarrow \pre AM_B, &m\otsbb b\mapsto m\cdot b\\
& j: \pre {B_M}(B_M\otss Y)_B \rightarrow \pre {B_M}(B_MY)_B, & b\otss y\mapsto b\cdot y\\
& k: \pre {B_M}(B_MY\otss B_M)_{B_M} \rightarrow \pre {B_M}(B_MY)_{B_M}, & \xi\otss b\mapsto \xi\cdot b
\end{align*} where $b\in B_M, m\in M, y\in Y,$ and $\xi\in B_MY.$ Let $U_{M'}$ be the composition of the $A-B_M$ correspondence isomorphisms
\[
\begin{tikzcd}[row sep=large, column sep=large]
&[-40pt]X\ots M' \arrow{r}{1_X\ot \iota^{-1}}  & X\ots M\otss B_M  \arrow{r}{U_M\ot 1_{B_M}}& M\otss Y\otss B_M\arrow{dll}[swap]{l^{-1}\ot 1_Y\ot 1_{B_M}} \\
& M'\otsm B_M\otss Y\otss B_M \arrow{r}[swap]{1_{M'}\ot j\ot 1_{B_M}} & M'\otsm B_MY\otss B_M \arrow{r}[swap]{1_{M'}\ot k} & M'\otsm B_M Y.
\end{tikzcd}
\]

 To prove  [$\pre {B_M}(B_M)_B, U_{B_M}$]$\circ$ $[\pre AM'_{B_M}, U_{M'}]=[\pre AM_B, U_M]$ we show that the diagram 
\[
\begin{tikzcd}
X\ots M'\otsm B_M \arrow{r}{1_X\ot l} \arrow[swap]{d}{(1_{M'}\ot U_{B_M})(U_{M'}\ot 1_{B_M})} & X\ots M \arrow{d}{U_M} \\
M'\otsm B_M \otss Y \arrow{r}\arrow{r}{l\ot 1_Y} & M\otss Y 
\end{tikzcd}
\]
commutes. Take an elementary tensor $x\ots m \in X\ots M'$. By Cohen-Hewitt factorization theorem there exist $m'\in M, b'\in B_M $ such that $m=m'\cdot b'.$  Then we have 
\begin{align*}
&U_{M'}(x\ots m)= (1_{M'}\ot k)(1_{M'}\ot j\ot 1_{B_M})(l^{-1}\ot 1_Y\ot 1_{B_M})(U_M\ot 1_{B_M})(1_X\ot \iota^{-1})(x\ots m)\\
&= (1_M\ot k)(1_M\ot j\ot 1_{B_M})(l^{-1}\ot 1_Y\ot 1_{B_M})(U_M\ot 1_{B_M})(x\ots m'\otss b') \\
&= (1_M\ot k)(1_M\ot j\ot 1_{B_M})(l^{-1}\ot 1_Y\ot 1_{B_M})\lim_{n\rightarrow\infty}\sum_{i=1}^{N_n} m_i^n\otss y_i^n\otss b',
\end{align*} where $\lim_{n\rightarrow\infty}\sum_{i=1}^{N_n} m_i^n\otss y_i^n= U_M(x\ots m').$ Again by Cohen-Hewitt factorization theorem, there exist $\xi_i^n\in M, c_i^n\in B_M$ such that $m_i^n=\xi_i^n\cdot c_i^n$. Then we may continue our computation as 
\begin{align*}
&= (1_M\ot k)(1_M\ot j\ot 1_{B_M})\lim_{n\rightarrow\infty}\sum_{i=1}^{N_n} \xi_i^n\otsbb c_i^n \otss y_i^n\otss b'\\
&=(1_M\ot k)\lim_{n\rightarrow\infty}\sum_{i=1}^{N_n} \xi_i^n\otsbb c_i^n \cdot  y_i^n\otss b'\\
&=\lim_{n\rightarrow\infty}\sum_{i=1}^{N_n} \xi_i^n\otsbb c_i^n \cdot y_i^n\cdot b'
\end{align*}
Then, for the elementary tensor $x\ots m\otss b$ of  \mbox{$X\ots M\otss B_M$} we have 
\begin{align*}
&\hspace{-1cm}(l\ot 1_Y)(1_M\ot U_{B_M})(U_{M'}\ot 1_{B_M})(x\ots m\otss b)\\
&= (l\ot 1_Y)(1_M\ot U_{B_M})\lim_{n\rightarrow\infty}\sum_{i=1}^{N_n} \xi_i^n\otsbb c_i^n \cdot  y_i^n\cdot b'\otss b\\
&=\lim_{n\rightarrow\infty}\sum_{i=1}^{N_n} m_i^n\otss y_i^n\cdot b'b\\
&= U_M(x\ots m')b'b\\
&= U_M(1_X\ot l)(x\ots m\otsbb b),
\end{align*}
as desired.\end{proof}

\begin{corollary}\label{left-full}
Let \tn{[\AMB, $U_M$]:} \AXA$\rightarrow$\BYB\ be a morphism in \cat. If \AMB\ is a left-full Hilbert bimodule, then so is the associated correspondence $\pre {\OX}(M\otss\OY)_{\OY}$. 
\end{corollary}

\begin{proof} By Proposition~\ref{factor} we have [$\pre {B_M}(B_M)_B, U_{B_M}$]$\circ$ [$\pre AM'_{B_M}, U_{M'}$]$=$[$\pre AM_B, U_M$], and thus \[ [\pre \OX(M\otss\OY)_{\OY}]=[\pre {\mathcal{O}_{B_MY}}(B_M\otss \OY)_{\OY}]\circ [\pre \OX(M'\otsbb \mathcal{O}_{B_M Y})_{\mathcal{O}_{B_MY}}].\]  Since $\pre AM'_{B_M}$ is an imprimitivity bimodule, [$\pre AM'_{B_M}, U_{M'}$] is an isomorphism in \cat, and thus, $\fun\left( [\pre AM'_{B_M}, U_{M'}]\right)= [\pre \OX(M'\otsbb \mathcal{O}_{B_MY})_{\mathcal{O}_{B_MY}}]$ is an isomorphism in \enchilada. This means  $\pre \OX(M'\otsbb \mathcal{O}_{B_MY})_{\mathcal{O}_{B_MY}}$ is an imprimitivity bimodule. We also know by Theorem~\ref{posinv} that $\pre {\mathcal{O}_{B_MY}}(B_M\otss \OY)_{\OY}$ is a left-full Hilbert bimodule. Hence, $\pre \OX(M\otss\OY)_{\OY}$ is a left-full Hilbert bimodule. \end{proof}

\section{Exactness}

We denote by \cats\ the subcategory of \cat\ where all objects are regular \C-correspondences. Every morphism in \cat\ has a kernel; however, not every morphism has a cokernel. We show in this section that every kernel in \cats\ has a cokernel. This observation leads us  to study exactness in the subcategory \cats\ instead of \cat .

To study kernels in \cats\ we need some understanding of monomorphisms in this category. Following Lemma is necessary for this purpose.  

\begin{lemma}\label{cancel} Let $\mu: M\otss N \rightarrow M'\otss N$ be an $A$--$C$-correspondence isomorphism where $M$ and $M'$ are $A$--$B$-correspondences, and \BNC\ is a left-full Hilbert bimodule. Then, there exists an  isomorphism $\iota:$ \AMB $\rightarrow$ \AMMB\ such that $\iota\ot 1_N$ = $\mu$. 
\end{lemma}

\begin{proof} 
Since \BNC\ is a left-full Hilbert bimodule, there exists a $C$--$B$-correspondence $\tilde{N}$ and a $B$--$B$-correspondence isomorphism \[j: N\otsc \tilde{N}\rightarrow B, \text{ } n_1\otsc \tilde{n_2}\mapsto \pre B\<n_1, n_2\>,\] where $n_1, n_2\in N.$ Define an isomorphism  $\iota:$ \AMB $\rightarrow$ \AMMB\ by 
\[ \iota= \xi_{(r,M',B)}\left(1_{M'}\ot j\right)\left(\mu\ot 1_{\tilde{N}}\right)\left(1_M\ot j^{-1}\right)\left(\xi_{(r,M,B)}\right)^{-1},\] where $\xi_{(r,M,B)}$ and $\xi_{(r,M',B)}$ are the $A$--$B$-correspondence isomorphisms defined as in Remark~\ref{notation}.
It suffices to use elementary tensors to verify the equality  $\iota\ot 1_N$ = $\mu$. Let $m'\in M'$ and  $n_1, n_2, n_3\in N$.  Then we have 
\begin{align*}
 \left(\xi_{(r,M',B)}\ot 1_N\right) &\left(1_{M'}\ot j\ot 1_N\right)(m'\otss n_1 \otsc \tilde{n_2}\otss n_3) \\
 &=  \left(\xi_{(r,M',B)}\ot 1_N\right) m'\otss \pre B\<n_1,n_2\>\otss n_3\\
&= m'\pre B\<n_1, n_2\>\otss n_3\\
&= m'\otss n_1\<n_2, n_3\>_C. 
\end{align*}
This shows that for any $x\in M\otss N,$ and $n,n'\in N$ we have 
 \[\left(\xi_{(r,M',B)}\ot 1_N\right)\left(1_{M'}\ot j\ot 1_N\right)\left(\mu\ot 1_{\tilde{N}}\ot 1_N\right)(x\otsc \tilde{n}\otss n')=\mu(x)\<n,n'\>_C\] and 
 \[ \mu\left(\xi_{(r,M,B)}\ot 1_N\right)\left(1_M\ot j \ot 1_N\right)(x\otsc \tilde{n}\otss n')= \mu(x)\<n,n'\>_C,\]
 as desired.\end{proof}

%\begin{proof} 
%It suffices to prove J_X\cdot A/I \subset J_{X/XI}. 

\begin{proposition}\label{mono} Let \tn{[\BNC, $U_N$]}: \BYB $\rightarrow$\CZC\ be a morphism in \cats . If \BNC\ is a left-full Hilbert bimodule, then \tn{[\BNC, $U_N$]} is a monomorphism in \cats . 
\end{proposition}

\begin{proof}   Let [\AMB, $U_M$], [\AMMB, \UMM]: \AXA $\rightarrow$ \BYB\  be morphisms in \cats\ satisfying 
\[ \text{[\BNC, $U_N$] $\circ$ [\AMB, $U_M$] = [\BNC, $U_N$] $\circ$ [\AMMB, \UMM]}.\] Then, there exists an isomorphism $\mu: M\otss N \rightarrow M'\otss N$ with the commutative diagram 
\[
\begin{tikzcd}
X\ots M\otss N \arrow{r}{1_X\ot \mu}  \arrow[swap]{d}{(1_M\ot U_N)(U_M\ot 1_N)} & X\ots M'\otss N \arrow{d}{(1_{M'}\ot U_N)(U_{M'}\ot 1_N)} \\
M\otss N\otsc Z \arrow{r}\arrow{r}{\mu\ot 1_Z} & M'\otss N\otsc Z.
\end{tikzcd}
\]
Since \BNC\ is a left-full Hilbert bimodule, by Lemma~\ref{cancel}, there exists an isomorphism $\iota:$ \AMB $\rightarrow$ \AMMB\ such that $\mu=\iota\ot 1_N$. We aim to show that the diagram 
\[
\begin{tikzcd}
X\ots M  \arrow{r}{1_X\ot \iota}  \arrow[swap]{d}{U_M} & X\ots M'  \arrow{d}{U_{M'}} \\
M\otss Y \arrow{r}\arrow{r}{\iota\ot 1_Y} & M'\otss Y
\end{tikzcd}
\]
commutes. 

By the first diagram above, we have 
\begin{align*}
(1_{M'}\ot U_N)(U_{M'}\ot 1_N)(1_X\ot \iota\ot 1_N) &= (\iota\ot 1_N\ot 1_Z)(1_M\ot U_N)(U_M\ot 1_N)\\
&= (1_{M'}\ot U_N)(\iota\ot 1_Y\ot 1_N)(U_M\ot 1_N), 
\end{align*}
which implies the equality 
\[(U_{M'}\ot 1_N)(1_X\ot \iota\ot 1_N) = (\iota\ot 1_Y\ot 1_N)(U_M\ot 1_N).\]
Since \BNC\ is an injective correspondence, by Lemma~\ref{fowler}, we have 
\[U_{M'}(1_X\ot \iota) = (\iota\ot 1_Y)U_M,\]
completing the proof. 
\end{proof}

\begin{remark}
Let  $U_M: X\ots M \rightarrow M\otss Y$ be an $A$--$B$-correspondence isomorphism. We know by Lemma~\ref{inv} that $\ker\varphi_M$ is a positive $X$-invariant ideal of $A$. And thus we may view $KX$ as a \C-correspondence over $K$, where $K$ denotes the ideal $\ker\varphi_M.$ Then, as described in Lemma~\ref{idealmorph}, $[\pre KK_A, U_K]$: \KX $\rightarrow$ \AXA\ is a morphism in \cats\ where \mbox{$U_K(kx\otsk k')= k\ots xk'$} , for any $k,k'\in K$ and $x\in X$. 
\end{remark}

We are now ready to determine kernels in \cats .

\begin{theorem}\label{kernel} Let  \tn{[\AMB,$U_M$]:} \AXA $\rightarrow$ \BYB\ be a morphism in \cats . Let $K$ denote the kernel of the homomorphism $\varphi_M: A\rightarrow \KK(M).$  Then, the object  $\pre K(KX)_K$ paired with the morphism $[\pre KK_A, U_K]:$ \KX $\rightarrow$ \AXA\ is a kernel of \tn{[\AMB,$U_M$]}. 

\end{theorem}

\begin{proof} 
We must show the following:
\begin{enumerate}
\item \tn{[\AMB, $U_M$]} $\circ$ [\KKA, $U_K$] = [0, $0_{KX,Y}$]; and 
\item assume [\CNA, $U_N$]: \CZC $\rightarrow$ \AXA\ is a morphism in \cats\ satisfying the equality \tn{[\AMB, $U_M$]} $\circ$ [\CNA, $U_N$] = [0, $0_{Z,Y}$] . Then, there exists a unique morphism [\CTK, $U_T$]: \CZC $\rightarrow$ \KX\ such that \mbox{[\KKA, $U_K$] $\circ$ [\CTK, $U_T$]=[\CNA, $U_N$].} 
\end{enumerate}

Item (1) is folklore. For (2), notice that  since $N\ots M \cong 0$ we have \mbox{$\<N,N\>_A\subset K$.}  Thus we may view \CNA\ as a $C$--$K$-correspondence, which we denote by $N'$. Moreover, we have the isomorphisms
\[ \iota: \pre CN\ots K_K \rightarrow \pre CN'_K  \hspace{2cm}  n\ots k \mapsto n\cdot k\]
and
\[ j: \pre CN'\otsk K_A \rightarrow \tn{\CNA} \hspace{2cm} n\otsk k \mapsto n\cdot k \]
for $n\in N$, $k\in K$. 
Now let $U_{N'}$ be the $C$--$K$-correspondence isomorphism
\[
\begin{tikzcd}[row sep=large, column sep=huge]
&[-50pt]\pre CZ\otsc N'_K \arrow{r}{1_Z\ot \iota^{-1}}  &\pre C Z\otsc N\ots K_K  \arrow{r}{U_N\ot 1_K}& \pre CN\ots X\ots K_K \arrow{dll}[swap]{j^{-1}\ot 1_X\ot 1_K} \\
& \pre CN'\ot_K K\ots X\ots K_K  \arrow{r}[swap]{1_{N'}\ot \xi_l\ot1_K} & \pre CN'\ot_K KX\ots K_K \arrow{r}[swap]{1_{N'}\ot \xi_r} & \pre CN'\ot_K KX_K,
\end{tikzcd}
\]
where $\xi_l$ is the $K$--$A$-correspondence isomorphism $\xi_{(l,X,K)}: K\ots X\rightarrow KX$, and $\xi_r$ is the $K$--$K$-correspondence isomorphism $\xi_{(r, KX,K)}: KX\ots K\rightarrow KX$, i.e., 
\[ U_{N'}:= [1_{N'}\ot \xi_r][1_{N'}\ot \xi_l \ot1_K][j^{-1}\ot 1_X\ot 1_K][U_N\ot 1_K][1_Z\ot\iota^{-1}].\]

We show that [$\pre C(N'\otsk K)_A$, $(1_{N'}\ot U_K)(U_{N'}\ot 1_K)$]=[\CNA, $U_N$], i.e.,  the diagram
\[
\begin{tikzcd}
Z\otsc N' \otsk K \arrow{r}{1_Z\ot j} \arrow[swap]{d}{(1_{N'}\ot U_K)(U_{N'}\ot 1_K)} & Z\otsc N \arrow{d}{U_N} \\
N'\otsk K\ots X \arrow{r}\arrow{r}{j \ot 1_X} & N\ots X
\end{tikzcd}
\]
commutes. Consider an elementary tensor  $n\ots x\ots k_1\otsk k_2$ of  \mbox{$\pre A(N\ots X\ots K\otsk K)_A$}. By Cohen-Hewitt factorization theorem we have $n=n'\cdot k'$ for some $n'\in N$ and $k'\in \<N,N\>_A\subset K$. Then,
\begin{align*}
(j \ot 1_X)(1_{N'}\ot U_K)&(1_{N'}\ot \xi_r \ot 1_K)(1_{N'}\ot \xi_l\ot 1_K\ot 1_K)(j^{-1}\ot 1_X\ot 1_K\ot 1_K)(n\ots x \ots k_1\otsk k_2)\\
&=(j \ot 1_X)(1_{N'}\ot U_K)(1_{N'}\ot \xi_r \ot 1_K)(n'\otsk k' x \ots k_1\otsk k_2)\\
&= (j \ot 1_X)(1_{N'}\ot U_K)(n'\otsk k'xk_1\otsk k_2)\\
&=  (j \ot 1_X)(n'\otsk k'\ots x k_1k_2)\\
&= n\ots x k_1k_2.
\end{align*}
On the other hand, it is not hard to seee that  
\[U_N(1_Z\ot j)(1_Z\ot \iota\ot 1_K)(U_N^{-1}\ot 1_K\ot 1_K) (n\ots x\ots k_1\otsk k_2) =n\ots xk_1k_2.\]

Uniqueness of the morphism [$\pre CN'_K$, $U_{N'}$] follows from Proposition~\ref{mono}, since \KKA\ is a left-full Hilbert bimodule. \end{proof}

%\begin{corollary} 
%If $[\pre AM_B, U_M]: \pre AX_A \rightarrow \pre BY_B$ is a kernel of $[\pre BN_C, U_N]: \pre BY_B\rightarrow\pre CZ_C$ in \cats\ then $B_M=\ker N.$
%\end{corollary}

We next study cokernels in \cats .

\begin{lemma}\label{epi} Let $\pre CX_B$ and $\pre CY_B$  be \C-correspondences. Let $\pre AC_C$ be the \C-correspondence where the left action is determined by the surjective map $\pi: A\rightarrow C$. If there exists an $A$--$B$-correspondence isomorphism \mbox{$U: C\otsc X \rightarrow C\otsc Y$,} then  there exists an isomorphism $V: \pre CX_B\rightarrow \pre CY_B$ such that $1_C\ot V=U$. \end{lemma}

\begin{proof} Consider the natural $A$--$B$-correspondence isomorphisms
\begin{align*}  
& \iota_{C,X}: C\otsc X\rightarrow X \text{} &c\otsc x\mapsto c\cdot x\\
& \iota_{C,Y}: C\otsc Y\rightarrow Y \text{} &c\otsc y\mapsto c\cdot y
\end{align*}
\cite[Lemma~3.3]{ench} tells us that the map $\iota_{C,Y}\circ U\circ \iota_{C,X}^{-1}: \pre AX_B\rightarrow \pre AY_B$ preserves the left $C$-module structure and thus, provides an isomorphism $\pre CX_B \rightarrow \pre CY_B$.  We observe that \mbox{$1_C\ot  \iota_{C,Y}U\iota_{C,X}^{-1}=U$}:  let $c,c'\in C, x\in X.$  Notice that since $U(c'\otsc x)=\lim_{n\rightarrow\infty}\sum_{i=1}^{N_n} c_i^n\otsc y_i^n$ for $c_i^n \in C, y_i^n\in Y$, we have 
\begin{align*}
(1_C\ot\iota_{C,Y})(1_C\ot U)(c\otsc c'\otsc x)&=c\otsc \lim_{n\rightarrow\infty}\sum_{i=1}^{N_n} c_i^n\cdot y_i^n\\
&= \lim_{n\rightarrow\infty}\sum_{i=1}^{N_n}  cc_i^n\otsc y_i^n \\
&= U(1_C\ot \iota_{C,X})(c\otsc c'\otsc x),
\end{align*} as desired. \end{proof}

\begin{proposition}\label{epic} Let \AXA\  be a regular \C-correspondence and let $I$ be an $X$-invariant ideal of $A$. Then, $[\pre A(A/I)_{A/I}, U_{A/I}]$: \AXA $\rightarrow \pre {A/I}(X/XI)_{A/I}$ is an epimorphism in \cats. 
\end{proposition}

\begin{proof} Assume there exist morphisms [\MI, $U_M$], [\NI, $U_N$]: $\pre {A/I}(X/XI)_{A/I}\rightarrow$\BYB\ in \cat\ such that 
\[\text{[\MI, $U_M$]$\circ$ [\AI, $U_{A/I}$] = [\NI, $U_N$]$\circ$ [\AI, $U_{A/I}$]}.\]
Then, there exists an $A$--$B$-correspondence isomorphism \[\mu: A/I \otsai M \rightarrow A/I\otsai N\] making the diagram
\[
\begin{tikzcd}
X\ots A/I\otsai M \arrow{r}{1_X\ot \mu}  \arrow[swap]{d}{(1_{A/I}\ot U_M)(U_{A/I}\ot 1_M)} & X\ots A/I \otsai N \arrow{d}{(1_{A/I}\ot U_N)(U_{A/I}\ot 1_N)} \\
A/I \otsai M\otss Y \arrow{r}\arrow{r}{\mu\ot 1_Y} & A/I\otsai N\ots Y
\end{tikzcd}
\]
commute.

Since the \C-correspondence \AI\ comes from the surjective homomorphism $A\rightarrow A/I$, by Lemma~\ref{epi},  there exists an isomorphism $\xi:$ \MI $\rightarrow$ \NI\ such that $\mu= 1_{A/I}\ot \xi$. Then, by the diagram above, we have
\begin{align*}
(1_{A/I}\ot\xi\ot 1_Y)(1_{A/I}\ot U_M)(U_{A/I}\ot 1_M) &= (1_{A/I}\ot U_N)(U_{A/I}\ot 1_N)(1_X\ot 1_{A/I}\ot\xi)\\
&= (1_{A/I}\ot U_N)(1_{A/I}\ot 1_{X/XI}\ot\xi)(U_{A/I}\ot 1_M),
\end{align*}
which means $1_{A/I}\ot (\xi\ot 1_Y)U_M = 1_{A/I}\ot U_N(1_{X/XI}\ot\xi)$. Since  [\AI] is an epimorphism in \enchilada , we may now conclude the equality  \mbox{$(\xi\ot 1_Y)U_M=U_N(1_{X/XI}\ot\xi)$,} which implies [\MI, $U_M$]=[\NI, $U_N$]. \end{proof}

\begin{theorem} Let  \tn{[\AMB,$U_M$]}: \AXA $\rightarrow$ \BYB\ be a morphism in \cats. If $B_M$ is a $Y$-invariant ideal of $B$, then a cokernel of  \tn{[\AMB,$U_M$]}  is $[\pre {B}(B/B_M)_{B/B_M}, U_{B/B_M}]$: \BYB $\rightarrow$ $\pre {B/B_M}(Y/YB_M)_{B/B_M}$. 
\end{theorem}

\begin{proof}
We must show the following:
\begin{enumerate}
\item  $[\pre B(B/B_M)_{B/B_M}, U_{B/B_M}]$  $\circ$ \tn{[\AMB, $U_M$]}  = [0, $0_{X,Y/YB_M}$]; and 
\item assume [\BNC, $U_N$]: \BYB $\rightarrow$ \CZC\  is a morphism in \cats\ satisfying the equality  \tn{[\BNC, $U_N$]} $\circ$ [\AMB, $U_M$] = [0, $0_{X,Z}$] . Then, there exists a unique morphism [$\pre {B/B_M}T_C$, $U_T$]: \YBM $\rightarrow$ \CZC\ satisfying the equality \mbox{[\BNC, $U_N$]$=$[$\pre {B/B_M}T_C$, $U_T$] $\circ$ [\BM, $U_{B/B_M}$].} 
\end{enumerate}

The first item is easy to verify. Let  [\BNC, $U_N$]: \BYB $\rightarrow$ \CZC\  be a morphism described as in the second item. Then \BNC\ can be viewed as a $B/B_M$--$C$-correspondence \cite[Lemma~3.3]{ench}, which we denote by $N'$. Now, let $\xi$ be the \C-correspondence isomorphism $\pre B(B/B_M\otsmm N')_C \rightarrow$ \BNC ; and consider the $B$--$C$-correspondence isomorphism
\[(\xi^{-1}\ot 1_Z)U_N(1_Y\ot \xi)(U_{B/B_M}^{-1}\ot 1_{N'}):  B/B_M\otsmm Y/YB_M\otsmm N'\rightarrow B/B_M\otsmm N'\otsc Z.\]
By Lemma~\ref{epi} there exists an isomorphism
\[U_{N'}:   \pre {B/B_M}(Y/YB_M\otsmm N')_C \rightarrow  \pre {B/B_M} (N'\otsc Z)_C\]
such that $1_{B/B_M}\ot U_{N'} = (\xi^{-1}\ot 1_Z)U_N(1_Y\ot \xi)(U_{B/B_M}^{-1}\ot 1_{N'})$. %Then we have 
%\begin{align*}
%(\xi\ot 1_Z)(1_{B/B_M}\ot U_T)(U_{B/B_M}\ot 1_N)&= (\xi\ot 1_Z) (\xi^{-1}\ot 1_Z)U_N(1_Y\ot \xi)(U_{B/B_M}\ot 1_N)^{-1}(U_{B/B_M}\ot 1_N)\\
%&= U_N(1_Y\ot \xi).
%\end{align*}
One can now see that the diagram
\[
\begin{tikzcd}
Y\otss B/B_M\otsmm   N' \arrow{r}{1_Y\ot \xi} \arrow[swap]{d}{(1_{B/B_M}\ot U_{N'})(U_{B/B_M}\ot 1_{N'})} & Y\otss N \arrow{d}{U_N} \\
B/B_M\otsmm N'\otsc Z  \arrow{r}\arrow{r}{\xi\ot 1_Z} & N\otsc Z
\end{tikzcd}
\]
commutes. The uniqueness of  [$\pre {B/B_M}N'_C$, $U_{N'}$] follows from Proposition~\ref{epic}.\end{proof}

%\begin{lemma} 
%If $[\pre AM_B, U_M]: \pre AX_A \rightarrow \pre BY_B$ is a kernel of $[\pre BN_C, U_N]: \pre BY_B\rightarrow\pre CZ_C$ in \cats\ then $B_M=\ker N.$
%\end{lemma}

Let [\AXB] be a morphism in \enchilada . A kernel of [\AXB] is the pair $(K, [\pre KK_A])$, where $K$ denotes the kernel of $\varphi_X: A\rightarrow \LL(X)$ \cite[Theorem~3.11]{ench}. A cokernel of [\AXB] is the pair $(B/B_X, [\pre B(B/B_X)_{B/B_X}])$ \cite[Corollary~3.12]{ench}.

\begin{definition}\label{def-ench}
A sequence $0\rightarrow A\xrightarrow{[\pre AX_B]} B \xrightarrow {[\pre BY_C]}C\rightarrow 0$ in \enchilada\ is \emph{exact}  if the pair $(A,[\pre AX_B])$ is a kernel of $[\pre BY_C]$ and the pair $(B, [\pre BY_C])$ is a cokernel of $[\pre AX_B]$. \end{definition}

\begin{proposition}\label{exact-enchilada}
A sequence $0\rightarrow A\xrightarrow{[\pre AX_B]} B \xrightarrow {[\pre BY_C]}C \rightarrow 0$ in \enchilada\ is exact if and only if the following three holds. 

\begin{enumerate}[\normalfont(1)]

\item $\pre AX_B$ is a left-full Hilbert bimodule;

\item $B_X=K$, where $K$ denotes the kernel of $\varphi_Y: B\rightarrow\LL(Y)$.

\item Hilbert $C$-module $Y$ viewed as a $B/K$-- $C$-correspondence $\pre {B/K}Y'_C$ is an \ibm.

\end{enumerate}

\end{proposition}

\begin{proof}

Assume we have (1)-(3). In \enchilada, we know that kernel of [\BYC] is the pair $(K, [\pre KK_B])$, where $K$ denotes the kernel of $\varphi_Y: B\rightarrow\LL(Y)$. On the other hand,  item (2) implies that [\BYC]$\circ$[\AXB]=[$\pre A0_C$]. Then, by the universal property of kernels there exists a morphism from $A$ to $K$ which [\AXB] factors through. As shown in  \cite[Theorem~3.9]{ench} this unique morphism is [$\pre AX'_{K}$] where $X$ is just $X'$ viewed as an $A$--$K$-correspondence. Since $\pre AX'_{K}$ is an \ibm\ we have that  [$\pre AX'_{K}$] is an isomorphism in \enchilada. It remains to show that  [\BYC] is a cokernel of [\AXB]. We know that a cokernel of  [\AXB] is [$\pre BB/K_{B/K}$], and since  [\BYC]$\circ$[\AXB]=[$\pre A0_C$], by the universal property of cokernels there exists a unique morphism which [\BYC] factors through. As shown in \cite[Proposition~3.11]{ench}, this unique morphism is [$\pre {B/K}Y'_C$], which is an isomorphism in \enchilada\ by item (3). 

For the other direction, assume [\AXB] is a kernel of [\BYC] and [\BYC] is a cokernel of [\AXB]. Since [\AXB] is a kernel of [\BYC], the correspondence $\pre AX_K$ in the kernel factorization $\pre AX_B\cong \pre A X\otsk K_B$ must be an \ibm, which means \AXB\ is a left-full Hilbert bimodule, giving us item (1). Moreover, $\pre AX_K$ being an \ibm\ implies that $B_X=K$, which proves item (2).  Since  [\BYC] is a cokernel of [\AXB], the correspondence $\pre {B/K}Y'_C$ in the cokernel factorization  $\pre BB/K\otimes_{B/K} Y'_C \cong \pre BY_C$ must be an \ibm, concluding the proof. \end{proof}

%first observe that we have \[X\otss Y\otsc \tilde{Y} =0 \iff X\otss \pre B\<Y,Y\>=0 \iff B_X\cap \pre B\<Y,Y\>=0.\] Coupling this with the fact that $B_X=K$ we get $K\subset \pre B\<Y,Y\>^\perp$. On the other hand, since  $\pre B\<Y,Y\>^\perp Y =  \pre B\<Y,Y\>^\perp \pre B\<Y,Y\> Y =0$ we have  $\pre B\<Y,Y\>^\perp\subset K,$ which allows us to conclude the equality  $\pre B\<Y,Y\>^\perp=K.$ 

\begin{definition}\label{def-eccor}
A sequence  \[0\rightarrow  \pre AX_A \xrightarrow{[\pre AM_B, U_M]} \pre BY_B \xrightarrow{[\pre BN_C, U_N]} \pre CZ_C \rightarrow 0 \] in \cats\ is called \emph{exact}  if the pair $\left(\pre AX_A, [\pre AM_B, U_M]\right)$ is a kernel of the morphism $[\pre BN_C, U_N]: \pre BY_B\rightarrow \pre CZ_C$; and the pair $\left(\pre BY_B, [\pre BN_C, U_N]\right)$ is a cokernel of the morphism $[\pre AM_B, U_M]: \pre AX_A\rightarrow \pre BY_B.$
\end{definition}

Note that $[\pre AM_B, U_M]$ being a kernel of $[\pre BN_C, U_N]$ in the sequence above implies that $B_M=\ker\varphi_N$. Then, Lemma~\ref{inv} allows us to conclude that $B_M$ is a $Y$-invariant ideal of $B$, and thus, cokernel of $[\pre AM_B, U_M]$ exists.

\begin{corollary}\label{exact-ECCOR} A sequence 
\[ 0\rightarrow \pre AX_A \xrightarrow{[\pre AM_B, U_M]} \pre BY_B \xrightarrow{[\pre BN_C, U_N]} \pre CZ_C \rightarrow 0 \]
is exact in \cats\ if and only if  the following holds.
\begin{enumerate}[\normalfont(1)]
\item $\pre AM_B$ is a left-full Hilbert bimodule;
\item $B_M=K$, where $K$ denotes the kernel of $\varphi_N: B\rightarrow \LL(N)$;
\item Hilbert $C$-module $N$ viewed as a $B/K$-- $C$-correspondence $\pre {B/K}N'_C$ is an imprimitivity bimodule.
\end{enumerate}
\end{corollary}

We omit the proof of Corollary~\ref{exact-ECCOR} since it can be shown by following the proof of Proposition~\ref{exact-enchilada}.

\begin{theorem}\label{exactness} The restriction of the functor $\fun$  to the category \cats\ is exact. 
\end{theorem}
\begin{proof}
Let the sequence 
\[ 0\rightarrow \pre AX_A \xrightarrow{[\pre AM_B, U_M]} \pre BY_B \xrightarrow{[\pre BN_C, U_N]} \pre CZ_C \rightarrow 0 \]
in \cats\ be exact. Denote $\ker\varphi_N$ by $K$.  We know that $\pre AM_B$ is a left-full Hilbert bimodule, the correspondence $\pre {B/K} N'_C$ is an \ibm, and we have the equality $B_M=K$. The functor $\fun$ maps this sequence to 
\[ 0\rightarrow \OX \xrightarrow{[\pre {\OX}(M\otss\OY)_{\OY}]} \OY \xrightarrow{[\pre {\OY}(N\otsc\OZ)_{\OZ}]} \OZ\rightarrow 0.\]
By Corollary~\ref{left-full}, we have that  $\pre {\OX}(M\otss\OY)_{\OY}$ is a left-full Hilbert bimodule. Since [$\pre {B/K} N'_C, U_{N'}$]: $\pre {B/K}Y/YK_{B/K}\rightarrow \pre CZ_C$  is an isomorphism in \cats,  we have  that $\fun([\pre {B/K} N'_C, U_{N'}])=[\pre {\mathcal O_{Y/YK}}(N'\otsc \OZ)_{\OZ}]$ is an isomorphism in \enchilada, and thus $\pre {\mathcal O_{Y/YK}}( N'\otsc \OZ)_{\OZ}$ is an \ibm. It remains to prove that $\< M\otss \OY, M\otss \OY\>_{\OY}=\ker\sigma$, where  $\sigma: \OY\rightarrow \KK(N\otsc\OZ)$ is the left action homomorphism associated to the correspondence $\pre {\OY}(N\otsc\OZ)_{\OZ}.$ Let $(\Upsilon, t)$  denote the universal covariant representation of \BYB . Then, by Proposition~\ref{kerneliso}, we have 
\[\< M\otss \OY, M\otss \OY\>_{\OY} =  \<\OY, K \cdot\OY\>  = \<\Upsilon(K)\> =\ker\sigma,\]
as desired. \end{proof}

\begin{example}
By using Theorem~\ref{exactness} we can easily see Theorem~\ref{quo} and Theorem~\ref{posinv} for the case when \AXA\ is a regular correspondence: let $I$ be an $X$-invariant ideal of $A$. Then, the sequence   \[0\rightarrow  \pre IIX_I \xrightarrow{[\pre II_A, U_I]} \pre AX_A \xrightarrow{[\pre AA/I_{A/I}, U_{A/I}]} \pre {A/I}X/XI_{A/I} \rightarrow 0 \] is exact in \cats.  And thus, the sequence 
\[ 0\rightarrow \pre  {\mathcal O_{IX}}\xrightarrow{[I\ots\OX]}\OX\xrightarrow{[A/I\ot_{A/I}  \mathcal O_{X/XI}]} \mathcal O_{X/XI}\rightarrow 0\] is exact in \enchilada. 
This implies by Proposition~\ref{exact-enchilada} that $\pre {\mathcal O_{IX}}(I\ots \OX)_{\OX}$ is a \mbox{left-full} Hilbert bimodule, which means  $\pre {\mathcal O_{IX}}(I\ots \OX)_{\<\Upsilon(I)\>}$ is an \ibm, where ($\Upsilon,t$) is the universal covariant representation of \AXA . Consequently, we have  $\mathcal O_{IX}\cong \KK(I\ots\OX) \cong \Upsilon(I)\OX\Upsilon(I)$.  On the other hand, again by Proposition~\ref{exact-enchilada}, we know that $\pre {\OX/\<\Upsilon(I)\>}(A/I\ot_{A/I} \mathcal O_{X/XI})_{ \mathcal O_{X/XI}}$ is an imprimitivity bimodule. This allows us to conclude the isomorphism $\OX/\<\Upsilon(I)\>\cong \KK(A/I\ot_{A/I} \mathcal O_{X/XI})\cong  \mathcal O_{X/XI}$.

\end{example}

%\nocite{*}

\end{document}